\documentclass[11pt]{article}
\usepackage{amsmath,amsthm,amssymb,enumerate,mathscinet}
\usepackage{fullpage}
\usepackage{graphicx}
\usepackage{color, url}
\usepackage{mathtools}
\usepackage{enumitem}
\usepackage{hyperref}
\usepackage{subfigure}

\newtheorem{theorem}{Theorem}[section]

\newtheorem{lemma}[theorem]{Lemma}

\newtheorem{claim}[theorem]{Claim}
\newtheorem{definition}[theorem]{Definition}

\newtheorem{thm}[theorem]{Theorem}

\newtheorem{conj}[theorem]{Conjecture}

\newtheorem{ques}[theorem]{Question} 
\numberwithin{equation}{section}
\def\comp{\textrm{comp}}
\def\Comp{\textrm{Comp}}
\def\picomp{\textrm{cpr}_\pi}
\def\A{\mathcal{A}}
\def\B{\mathcal{B}}
\def\C{\mathcal{C}}
\def\D{\mathcal{D}}
\def\C{\mathcal{C}}
\def\F{\mathcal{F}}
\def\G{\mathcal{G}}
\def\HH{\mathcal{H}}
\def\cL{\mathcal{L}}
\def\P{\mathcal{P}}
\def\S{\mathcal{S}}

\def\X{\mathcal{X}}
\def\Y{\mathcal{Y}}
\def\Z{\mathcal{Z}}
\def\w{\omega}

\def\eps{\varepsilon}

\def\rk{\text{rk}}
\newcommand{\RSU}{\mathbf{P}}
\def\bC{\mathbf{C}}
\newcommand{\N}{\mathbb{N}}
\newcommand{\floor}[1]{\left\lfloor#1\right\rfloor}

\def\COMMENT#1{}
\let\COMMENT=\footnote% COMMENT OUT for clean output

\title{Families in posets minimizing the number of comparable pairs } 

\author{J\'ozsef Balogh\footnote{Department of Mathematical Sciences,
 University of Illinois at Urbana-Champaign, Urbana, Illinois 61801, USA, {\tt
jobal@math.uiuc.edu}. The first author is partially supported by NSF Grant DMS-1500121, Arnold O. Beckman Research Award (UIUC Campus Research Board 15006) and by the Langan Scholar Fund (UIUC).},
~\v{S}\'{a}rka Pet\v{r}\'{i}\v{c}kov\'{a}\footnote{University of Illinois at Urbana-Champaign, Urbana, Illinois 61801, USA, {\tt
petrckv2@illinois.edu}.}
 ~and Adam Zsolt Wagner\footnote{University of Illinois at Urbana-Champaign, Urbana, Illinois 61801, USA, {\tt
zawagne2@illinois.edu}. }}
\begin{document}
\maketitle
\begin{abstract}
Given a graded poset $P$ we say a family $\F\subseteq P$ is \emph{centered} if it is obtained by `taking sets as close to the middle layer as possible'. A poset $P$ is said to have the \emph{centeredness property} if for any $M$, among all families of size $M$ in $P$, centered families contain the minimum number of comparable pairs. Kleitman showed that the Boolean lattice $\{0,1\}^n$ has the centeredness property. It was conjectured by Noel, Scott and Sudakov, and by Balogh and Wagner, that the poset $\{0,1,\ldots,k\}^n$ also has the centeredness property, provided $n$ is sufficiently large compared to $k$. We show that this conjecture is false for all $k\geq 2$ and investigate the range of $M$ for which it holds. Further, we improve a result of Noel, Scott and Sudakov by showing that the poset of subspaces of $\mathbb{F}_q^n$ has the centeredness property. Several open questions are also given.
\end{abstract}

\section{Introduction}

Given a poset $P$, we say that two elements $A,B\in P$ form a \emph{comparable pair} if $A\leq B$ or $B\leq A$. The study of families of sets containing few comparable pairs started with Sperner's Theorem, a cornerstone result of combinatorics. It states that the largest antichain (i.e.~family containing no comparable pairs) in the Boolean lattice $\P(n)=\{0,1\}^n$ has size $\binom{n}{\lfloor n/2 \rfloor}$. 
The following natural question was first posed by Erd\H{o}s and Katona for $r=2$ and then extended by Kleitman~\cite{kleitman} some fifty years ago: Given a poset $\P(n)$ and an integer $M$, what is the minimum number of $r$-chains that a family of $M$ elements in $\P(n)$ must contain? For $r=2$, the case of comparable pairs, the question was completely resolved by Kleitman~\cite{kleitman}. For $r\geq 3$, we refer the reader to~\cite{baloghwagner,dasgansudakov,griggs}. Here we are interested in the case $r=2$, but for a general poset $P$.
 
% \begin{problem}
% Given a poset $P$ and an integer $M$, what is the minimum number of comparable pairs that a family of $M$ elements in $P$ must contain?
%\end{problem}

\subsection{Centered families in $\{0,1,\ldots,k\}^n$}

We say that a family $\F \subseteq \{0,1\}^n$ is \emph{centered} if for any two sets $A,B\in \{0,1\}^n$ with $A\in\F$ and $B\notin \F$ we have that 
$$\bigg||A|-\frac{n}{2}\bigg|\leq\bigg||B|-\frac{n}{2}\bigg|,$$ where $|A|$ denotes the number of $1$-coordinates in $A$. That is, $\F$ is centered if it is constructed by ``taking sets that are as close to the middle layer as possible''. This same notion can be extended to the poset $\{0,1,\ldots,k\}^n$ where $A\leq B$ if $A_i\leq B_i$ for all $i\in [n]$, where $A_i$ and $B_i$ are the $i$th coordinates of $A$ and $B$.  We say that a family $\F\subseteq \{0,1,\ldots,k\}^n$ is \emph{centered} if for any two vectors $A, B \in \{0,1,\ldots,k\}^n$ with $A \in \F$ and $B \notin \F$ we
have that $$\bigg|\sum_{i=1}^n A_i - \frac{nk}{2}\bigg| \leq \bigg|\sum_{i=1}^n B_i- \frac{nk}{2}\bigg|.$$  

While elements of $\{0,1,\ldots,k\}^n$ are vectors of length $n$ they can also be thought of as subsets of an $n$-element multiset where each element has multiplicity $k$. For this reason in what follows we will use the words ``vector'' and ``set'' interchangeably when referring to elements of $\{0,1,\ldots,k\}^n$.

This definition of centeredness defined above for the posets $\{0,1,\ldots,k\}^n$  has a natural extension to the family $\mathbf{P}$ of graded posets that satisfy some necessary properties, we give details of this in Section~\ref{gradedposetsdef}. Denote by $\comp(\F)$ the number of comparable pairs in $\F$. Given a poset $P\in\RSU$ and a positive integer $M$ we say that a family $\F\subseteq P$ of size $M$ is $M$-\emph{optimal} if for all families $\F'\subseteq P$ of size $M$ we have $\comp(\F)\leq\comp(\F')$. A poset $P\in\RSU$ has the \emph{centeredness property} if for all $M\leq |P|$ there exists an $M$-optimal centered family.
%The poset $P=\{0,1,\ldots,k\}^n$ has the \emph{centeredness property} if for all $M\leq |P|$, among all families $\F\subseteq P$ of size $M$, the number of comparable pairs contained in $\F$ is minimized if $\F$ is a centered family. Note that this definition does not require that centered families be the only families of given size minimizing the number of comparable pairs.
 Using this terminology, Kleitman's celebrated theorem~\cite{kleitman} from 1968 can be stated as follows:
\begin{thm}[Kleitman~\cite{kleitman}]\label{kleitmanthm}
The poset $\{0,1\}^n$ has the centeredness property for all $n\in \N$.
\end{thm}

In~\cite{dasgansudakov} the authors characterised precisely which families achieve the minimum number of contained comparable pairs. It is natural to ask whether an extension of Theorem~\ref{kleitmanthm} holds for the poset $\{0,1,\ldots,k\}^n$ with $k\geq 2$ as well. It was showed in~\cite{baloghwagner} that there exists a counterexample with $n=2$ and  $k=16$. The following conjecture was raised independently in~\cite{noelscottsudakov} and~\cite{baloghwagner}:
\begin{conj}[\label{stupidconj}Noel--Scott--Sudakov~\cite{noelscottsudakov}, Balogh--Wagner~\cite{baloghwagner}]
For every $k$ there exists an $n_0$ such that if $n\geq n_0$, then the poset $\{0,1,\ldots,k\}^n$ has the centeredness property.
\end{conj}

Our main result is the construction of two different classes of explicit counterexamples to this natural generalisation of Theorem~\ref{kleitmanthm}. We show that for every $k$, if $n$ is sufficiently large, then there exists a suitable choice of $M$ and a family $\F$ of size $M$ that contains strictly fewer comparable pairs than the centered families of the same size.

Denote by $\cL_r(n,k)$ the \emph{$r$-th layer} of $\{0,1,\ldots,k\}^n$, i.e.~the set of vectors in $\{0,1,\ldots,k\}^n$ whose coordinates sum to $r$, and let $\ell_r(n,k)\coloneqq  |\cL_r(n,k)|$. Write $\Sigma_r(n,k)$ for the total size of the $r$ middle layers of $\{0,1,\ldots,k\}^n$. For $M\leq \Sigma_1(n,k)$ there exists an antichain of size $M$ in the middle layer $\cL_{\floor{nk/2}}(n,k)$  and hence Conjecture~\ref{stupidconj}  trivially holds. 

Our main result for the poset $\{0,1,2\}^n$ is the following. 

\begin{thm}\label{mainresult}
\begin{enumerate}[itemsep=2pt,parsep=2pt,topsep=2pt,partopsep=2pt]
\item[(a)] Let $\eps>0$,  $n$ be sufficiently large, and $M\leq (1-\eps)\Sigma_3(n,2)$. Then there exists an $M$-optimal  centered family in $\{0,1,2\}^n$.
\item[(b)] Let $n$ be sufficiently large and $M=\Sigma_6(n,2)-\binom{n}{3}-1$. Then none of the centered families in $\{0,1,2\}^n$ are $M$-optimal.
\end{enumerate}
\end{thm}

Theorem~\ref{mainresult} says that the smallest $M=M_0$ for which  Conjecture~\ref{stupidconj} breaks down (for $k=2$) satisfies $(1-\eps)\Sigma_3(n,2) < M_0 < \Sigma_6(n,2)-\binom{n}{3}$. For $k=2$ and $M$ slightly larger than $\Sigma_1(n,2)$ it was previously shown by Noel--Scott--Sudakov~\cite{noelscottsudakov} that centered families contain asymptotically the optimal number of comparable pairs. They also obtained good lower bounds for the number of comparable pairs in larger families. 
\begin{thm}[Noel--Scott--Sudakov~\cite{noelscottsudakov}]
Let $r$ be a fixed positive integer. Then there exists a constant $n_0(r)$ such that if $n \geq n_0(r)$ and $\F \subseteq \{0, 1, 2\}^n$ has cardinality at least $\Sigma_r(n,2)+t$ then 
$$\emph{comp}(\F)\geq\left(\frac{\ell_{3r-1}(n,2)}{\ell_{2r-1}(n,2)}-1\right)t.$$
\end{thm}
While at first sight it may seem feasible that Conjecture~\ref{stupidconj} holds for much larger $M$, Theorem~\ref{largeresult}~ shows that this is not the case.

\begin{thm}\label{largeresult}
Let $k\ge 2$ and $\eps>0$. There exists a constant $n_0=n_0(k,\eps)$ such that for every $n\ge n_0$, if $M= \Sigma_j(n,k)$, where $(1+\eps)\log_2n \leq j \leq \sqrt{n}/\log_2 n$, then  none of the centered families in $\{0,1,\ldots,k\}^n$ are $M$-optimal.
\end{thm}

%For the corresponding maximization question, i.e.,~determining the maximum possible number of comparable pairs amongst families of size $M$ in $\P(n)$, we refer the reader to \cite{alondasglebovsudakov}. \textcolor{red}{This does not fit anywhere}

\subsection{Centered families in other posets}\label{gradedposetsdef}

The notion of centeredness can be readily extended to several other common posets that satisfy some nice properties. In a poset $P$, \emph{$y$ covers $x$} if $x<y$ and there is no element $z$ such that $x<z<y$. We say that the poset $P$ is a \emph{graded poset} if it is equipped with a rank function  $\text{rk:}\, P\rightarrow \mathbb{N}\cup \{0\}$ which satisfies that $\rk(x)<\rk(y)$ whenever $x<y$, and $\rk(y)=\rk(x)+1$ whenever $y$ covers $x$. The \emph{rank of a poset} $P$ is the maximum rank of an element of $P$. Given a graded poset $P$, the $r$-th layer $\cL_r(P)$ is the collection of elements in $P$ of rank $r$, $\ell_r(P)$ is the size of $\cL_r(P)$, and $\Sigma_r(P)$ is the total number of elements of $P$ in the middle $r$ layers.  A graded poset of rank $n$ is \emph{rank-symmetric} if $\ell_i(P)=\ell_{n-i}(P)$ for $0\leq i\leq n$ and it is \emph{rank-unimodal}  if $\ell_0(P)\leq\ldots\leq \ell_j(P) \geq \ell_{j+1}(P) \geq\ldots\geq\ell_n(P)$ for some $0\leq j\leq n$. Denote by $\RSU$ the family of all graded posets that are rank-symmetric and rank-unimodal, and by $\RSU(n)$ the posets in $\RSU$ of rank $n$.
 
 We will extend the notion of centeredness only to the posets in $\RSU$. Note that every $P\in \RSU(n)$ satisfies that its largest layer is $\cL_{\lfloor n/2\rfloor}(P)$ and its $k$ largest layers are the $k$ layers closest to the middle layer.  Examples of such posets include $\{0,1,\ldots,k\}^n$ where $(A_1,\dots, A_n)\leq (B_1,\dots, B_n)$ if $A_i\leq B_i$ for all $1\leq i\leq n$, and the poset $\mathcal{V}(q,n)$ of subspaces of $\mathbb{F}_q^n$ ordered by inclusion where $q$ is a prime power.

Similarly as before, given a poset $P\in\RSU(n)$, we say that a family $\F \subseteq P$ is \emph{centered} if for any two sets $A,B\in P$ with $A\in\F$ and $B\notin \F$ we have that their ranks $\rk(A), \rk(B)$ satisfy
$$\bigg|\rk(A)-\frac{n}{2}\bigg|\leq\bigg|\rk(B)-\frac{n}{2}\bigg|.$$ In other words, $\F$ is centered if it is constructed by ``taking sets that are as close to the middle layer as possible''. Note that if $P=\{0,1,\ldots,k\}^n$, then this definition is the same as the definition of `centered' introduced in the previous section (where the rank of $P$ was $nk$). 

Consider now for a prime power $q$ the poset $\mathcal{V}(q,n)$ of subspaces of $\mathbb{F}_q^n$ ordered by inclusion. Denote by ${n\brack i}_q$ the number of subspaces of $\mathbb{F}_q^n$ of dimension $i$. Note that ${n\brack i}_q=\prod_{j=0}^{i-1} \frac{1-q^{n-j}}{1-q^{j+1}}$. 
The following result of Noel, Scott, and Sudakov~\cite{noelscottsudakov} provides a lower bound on $\comp(\F)$ for $\F\subseteq\mathcal{V}(q,n)$.

\begin{thm}[Noel--Scott--Sudakov~\cite{noelscottsudakov}]
Let $q$ be a prime power and $k$ be a fixed positive integer. There exists a constant $n_0(k)$ such that for $n\geq n_0(k)$ and $\F\subseteq \mathcal{V}(q,n)$, 
$$\mbox{If }\ \ |\F|\ge \sum_{r=0}^{k-1}{n \brack \lceil\frac{n-k+1+2r}{2}\rceil}_q+t,\ \ \ \mbox{ then }\ \ \ \emph{comp}(\F)\geq t{\lceil (n+k)/2\rceil\brack k}_q.$$
\end{thm}

They pointed out that this bound is attained by a centered family and hence best possible when $k=1$ and $0\leq t\leq {n \brack \lfloor (n-1)/2\rfloor}_q$. We show that centered families are best for all sizes.
\begin{thm}\label{subspacebs}
Let $q$ be a prime power and $n\geq 1$. Then the poset $\mathcal{V}(q,n)$ has the centeredness property.
\end{thm}
%%
%Our proofs of Theorem~\ref{mainresult}~(a) and Theorem~\ref{subspacebs} are heavily based on the compression techniques of Kleitman~\cite{kleitman}. The proof of Theorem~\ref{mainresult}~(b) arose when we attempted to prove that Conjecture~\ref{stupidconj} holds in the range $M\leq \Sigma_{(1-\eps)\log_2n}(n,2)$ -- all our proof attempts kept breaking down and they eventually led us to this counterexample. The construction in Theorem~\ref{largeresult} came from the observation that for large enough $M$, centered families are not even locally optimal, and in fact by replacing one of its elements in an appropriate way we can decrease the number of comparable pairs in the family.

Our proofs of Theorem~\ref{mainresult}~(a) and Theorem~\ref{subspacebs} are heavily based on the compression techniques of Kleitman~\cite{kleitman}. The construction in Theorem~\ref{largeresult} came from the observation that for large enough $M$, centered families are not even locally optimal, and in fact by replacing one of their elements in an appropriate way we can decrease the number of comparable pairs in the family.

For the corresponding maximization question, i.e.~determining the maximum possible number of comparable pairs amongst families of size $M$ in $\P(n)$ we refer the reader to \cite{alondasglebovsudakov}.

\section{Proof of Theorem~\ref{mainresult} (a)}\label{sectionmaina}

In this section we will apply the compression techniques of Kleitman~\cite{kleitman} to prove that for small $M$ there exist $M$-optimal centered families in $\{0,1,2\}^n$. 
Whenever $A=(A_1, \dots, A_n)$ is an element of $\{0,1,2\}^n$, we will define the \emph{size} (or \emph{rank}) of $A$ by $|A|\coloneqq \sum_{i=1}^n A_i$. We will use $a_0,a_1$ and $a_2$ to denote the number of $0$-, $1$-, and $2$-coordinates of $A$ (that is, $a_i\coloneqq |\{j:A_j=i\}|$). Similarly for $B\in\{0,1,2\}^n$ we will use the variables $b_0,b_1,b_2$ in the same fashion. The \emph{complement} of a set $A\in\{0,1,2\}^n$ is defined as $A^c\coloneqq (2-A_1,\dots,2-A_n)$. For a permutation $\pi\in S_n$ and a set $A\in\{0,1,2\}^n$ we denote by $\pi(A)$ the set $(A_{\pi(1)}, \dots, A_{\pi(n)})$. For a family $\F\subseteq\{0,1,2\}^n$ and integer $0\leq r\leq 2n$, we write $\F_r=\{A\in \F: |A|=r\}$ and $N_r(A)\coloneqq \{B: |B|=r,  B\subseteq A \text{ or } A\subseteq B\}$. Recall that in the poset $\{0,1,2\}^n$, $\cL_r(n,2)$ denotes the $r$-th layer  and $\Sigma_j(n,2)$ the total size of the $j$ middle layers. In this section, we will often shorten $\cL_r(n,2)$ to $\cL_r$ and $\Sigma_j(n,2)$ to $\Sigma_j$. Recall also that a family $\F\subseteq \{0,1,2\}^n$ of size $M$ is called $M$-optimal if there is no other family $\F'\subseteq \{0,1,2\}^n$ of size $M$ that contains strictly fewer comparable pairs than $\F$. Our goal is to show that there exists an $M$-optimal family that is centered.

Let $\eps>0$, let $n$ be sufficiently large so that all the following estimates hold, and fix an $M\leq(1-\eps)\Sigma_3(n,2)$. The proof is by induction on $M$, with the base case $M\le \Sigma_1(n,2)$ in which case there is an antichain in $\cL_n$ of size $M$ and the claim follows. Hence we will assume that there exists an $(M-1)$-optimal centered family, and show that there exists an $M$-optimal centered family. Our first goal is to show that there exist $M$-optimal families that are contained in the middle three layers of $\{0,1,2\}^n$.

%To finish the proof of Theorem~\ref{mainresult}~(a) we will then show that one of these $M$-optimal families fully contains the middle layer.  
The following claim will be useful for us:

\begin{claim}\label{claimfuncond}
If $A, B\in \{0,1,2\}^n$ such that $B\subseteq A$ and $|B|\geq n$, then for every $i\in \{1,\dots,|A|-|B|\}$,
$$|N_{|B|+i}(B)|\le |N_{|A|-i}(A)|.$$

\end{claim}
\begin{proof}

Suppose that $|A|,|B|\geq n$. We show that $B^c$ has at most as many $2$'s and at least as many $0$'s as $A$. This implies that there exists a permutation $\pi(B^c)$ of the coordinates of $B^c$ such that $\pi(B^c)\subseteq A$.  Thus, $\pi(B^c)$ has at most as many neighbors in level $\cL_{|\pi(B^c)|-i}$ as $A$ does in level $\cL_{|A|-i}$, for every $i\in \mathbb{N}\cup \{0\}$, so $$|N_{|B|+i}(B)|=  |N_{|B^c|-i}(B^c)| = |N_{|B^c|-i}(\pi(B^c))| \le |N_{|A|-i}(A)|.$$

The number of $0$'s in $B^c$ is equal to $b_2$ and the number of $2$'s in $B^c$ is equal to $b_0$. Hence we want to show that $b_0\leq a_2$ and $b_2\geq a_0$. Note first that since $B\subseteq A$, we have $b_2\leq a_2$ and $b_0\geq a_0$.

Let $k,l$ be such that $|A|=n+k$ and $|B|=n+l$. From $|A|>|B|\geq n$ we have that $k>l\geq 0$.
Since $a_0+a_1+a_2=n$ and $a_1+2a_2=n+k$, we have $a_2-a_0=k$, and similarly $b_2-b_0=l$. 
$$
b_0 = b_2-l \le a_2-l \le a_2\ \ \ \ \mbox{ and }\ \ \ \  b_2 =b_0+l\ge b_0 \ge a_0.
$$
\end{proof}

A family $\F$ in a poset $P\in\RSU$ is \emph{compressed} if for every element $A\in \F$, every element comparable with $A$ that is closer to the middle than $A$ is in $\F$. Kleitman proved that every family in the Boolean lattice ``can be compressed'' without increasing the number of comparable pairs. It is not clear why this would be the case for $\{0,1,\dots, k\}^n$ with $k>2$. 
In the poset $\{0,1,2\}^n$ we can however at least obtain an analogous result for a weaker notion of top- and bottom-compressed, given in the following definition.

\begin{definition}\normalfont
A family $\F\subseteq \{0,1,2\}^n$ is \emph{top-compressed} if the following condition holds:
\begin{itemize}[itemsep=2pt,parsep=2pt,topsep=2pt,partopsep=2pt]
\item[(T)] If $A\in\F$ with $|A|> n$ and $B\subseteq A$ with $|B|\ge n$, then $B\in\F$.
\end{itemize}
A family $\F\subseteq \{0,1,2\}^n$ is \emph{bottom-compressed} if the following condition holds:
\begin{itemize}[itemsep=2pt,parsep=2pt,topsep=2pt,partopsep=2pt]
\item[(B)] If $A\in\F$ with $|A|< n$ and $B\supseteq A$ with $|B|\le n$, then $B\in\F$.
\end{itemize}

%\item[(C2)] If $A$ is a minimal element of $\F$ with $|A|\leq n-2$ and $B\supset A$ is such that $|B|\leq |A|+3$ then $B\in\F$.
\end{definition}

\begin{lemma}\label{l_compressed}
For every natural number $K\le 3^n$, there exists a $K$-optimal family that is top- and bottom-compressed.
\end{lemma}

\begin{proof}
Let $\F$ be a $K$-optimal family and suppose that condition (T) is violated. Handling the case when (B) is violated is the same. Let $(A,B)$ be such a violating set-pair that satisfies that whenever $(A',B')$ is another set-pair violating condition (T) we have
\begin{itemize}
\item $|A|>|A'|$, or
\item $|A|=|A'|$ and $|B|\geq |B'|$.
\end{itemize}  
Let $a=|A|$ and $b=|B|$ and note that all elements in levels $\cL_{b+1}, \dots, \cL_{a-1}$ that are comparable with $A$ are in $\F$.

Whenever $\X$ is a family in $\{0,1,2\}^n$ and $C\in\{0,1,2\}^n$ is any element of $\{0,1,2\}^n$ we write $N_{\X}(C)$ for the set of elements in $\X$ comparable with $C$. Additionally, for all $C\in\{0,1,2\}^n$ and $r\in \{0,1,\ldots,2n\}$ we let $N_r(C)\coloneqq N_{\cL_r}(C)$ and $N (C)\coloneqq N_{\{0,1,2\}^n}(C)$. Finally, whenever $\X,\Y$ are any two (not necessarily disjoint) families in $\{0,1,2\}^n$ we write $N_{\X}(\Y):=\bigcup_{Y\in\Y} N_{\X}(Y)$.

 We will show that we can iteratively replace some elements of $\F_{a}=\F\cap \cL_a$ by elements of $\overline{\F_b}=\cL_b \setminus \F$ without increasing the number of comparable pairs. We will consider several cases based on sizes of $\F_{a}$ and $\overline{\F_b}$ and the existence of ``good'' matchings that allow us to top-compress $\F$. Since $n\leq b<a$, the total value $\sum_{C\in \F} ||C|-n|$ of the family strictly decreases, ensuring that this process will terminate.

Form a bipartite graph with parts $\F_{a}$ and $\overline{\F_{b}}$ and with edges between comparable pairs. Our goal will be to find non-empty families $\A\subseteq \F_{a}$ and $\B\subseteq \overline{\F_{b}}$ with $|\A|=|\B|$ such that there is a perfect matching $f$ between $\A$ and $\B$. Note that for any such families $\A,\B$ the family $\G=(\F\setminus \A)\cup\B$ has the same size as $\F$; hence if we can pick the pair $(\A,\B)$ such that $\G=(\F\setminus \A)\cup\B$ has not more comparable pairs than $\F$ then we may replace $\F$ by $\G$. Hence Lemma~\ref{l_compressed} follows from the following claim:

\begin{claim}
There exist non-empty families $\A\subseteq  \F_{a}$ and $\B\subseteq \overline{\F_{b}}$ with $|\A|=|\B|$ such that 
\begin{itemize}
\item there is a perfect matching $f$ between $\A$ and $\B$, and
\item the number of comparable pairs in $\F$ is no less than the number of comparable pairs in $\G=(\F\setminus \A)\cup\B$.
\end{itemize}
\end{claim}
\begin{proof}
Given any $\A\subseteq \F_{a}$ and $\B\subseteq \overline{\F_{b}}$ with $|\A|=|\B|$, we pick two sets $A'\in\A$ and $B'\in \B$ with $B'\subset A'$ arbitrarily. We compare the sizes of neighborhoods of $\A$ and $\B$ in $\F$, in the following four parts of the poset $\{0,1,2\}^n$:
\begin{enumerate}[itemsep=2pt,parsep=2pt,topsep=2pt,partopsep=2pt]
\item In levels $\cL_{a+1}, \dots, \cL_{2n}$: if there was a set $A^*\in\F$ in one of these layers that is comparable to $B'$ then the pair $(A^*,B')$ would violate condition (T). Since $|A^*|>|A|$ this is not possible as $(A,B)$ was chosen amongst violating pairs so that $|A|$ is as large as possible.
\item In levels $\cL_{0}, \dots, \cL_{b-1}$: Since $B'\subseteq A'$, whenever $C\in\{0,1,2\}^n$ is such that $C\subseteq B'$, we also have $C\subseteq A'$.
%If $(B,C)\in \Comp (\G-\F)$ is a new comparable pair such that $|C|<|B|$, then $(B,C) \in \Comp(\F-\G)$.
\item In levels $\cL_{b+1}, \dots, \cL_{a-1}$: Since all elements in these levels that are comparable with $A'$ are in $\F$, by Claim~\ref{claimfuncond}, for every $i\in [a-b-1]$,
$$|\F\cap N_{b+i}(B')|\le |N_{b+i}(B')|  \le |N_{a-i}(A')| = |\F\cap N_{a-i}(A')|.$$ Thus, every element $B'\in \B$ has at most as many neighbors in $\cL_{b+1}\cup \dots \cup \cL_{a-1}$ as every $A'\in \A$ does. 
\item In levels $\cL_{a}$ and $\cL_{b}$, the proof splits into several cases. %In one case we apply Claim~\ref{claimfuncond} with $i=a-b-1$.
\end{enumerate}
 In each case below, we present suitable families $\A\subset \F_a$ and $\B\subset \overline{\F_b}$ with a perfect matching $f$ between $\A$ and $\B$ for which
\begin{equation}\label{eq_Levels_a_b}
e(\B, \G_a) \le e(\A, \F_b),
\end{equation}
where $e(\C,\D)$ denotes the number of edges between the families $\C$ and $\D$, and $\G_a=\F_a\setminus \A$. Note that if $\A,\B$ satisfies~(\ref{eq_Levels_a_b}) then by the above, the family $\G = \F\cup \B \setminus \A$ has at most as many comparable pairs as $\F$ does.

Suppose first that there exists a matching $f$ between $\F_{a}$ and $\overline{\F_b}$ covering $\F_{a}$. Let $\A=\F_{a}$ and $\B=f(\F_{a})$. Then $\G_a=\F_a\setminus \A = \emptyset$, so $e(\B, \G_a)=0$ and~(\ref{eq_Levels_a_b}) is satisfied for this choice of $\A,\B$. Henceforth we assume that there is no matching $f$ between $\F_{a}$ and $\overline{\F_b}$ covering $\F_{a}$, and we restrict our attention to the bipartite graph between vertex sets $(\X, \Y)$, where $$\X\coloneqq \F_{a}\ \  \mbox{ and }\ \ \Y\coloneqq N(\F_{a}) \cap \overline{\F_b},$$with edges between comparable pairs.

\textbf{Case 1:} $|\X|\leq |\Y|$.
By Hall's theorem, since there is no matching between $\X$ and $\Y$ covering $\X$, there must be a vertex set $\X_0\subseteq \X$ such that $|N_{\Y}(\X_0)|< |\X_0|$. Choose $\X_0$ to be a maximal such vertex set. Then there must exist a matching $f$ between $\X\setminus\X_0$ and $\Y \setminus N_{\Y}(\X_0)$ covering $\X\setminus\X_0$. Define $\A=\X \setminus \X_0$ and $\B=f(\X\setminus\X_0)$. Since there is no edge between $\B=f(\X\setminus\X_0)$ and $\G_a= \X_0$, the relation (\ref{eq_Levels_a_b}) holds.

\begin{figure}
\centering     %%% not \center
\subfigure[Case 1.]{\label{fig1} \includegraphics[scale=0.85]{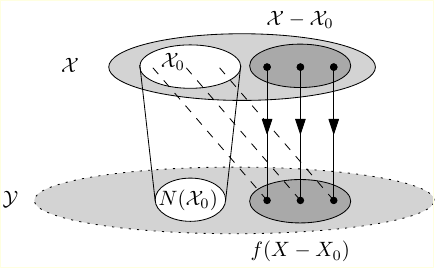}}
\hspace{1cm}
\subfigure[Case 2. There exists a matching between $\X$ and $\Y$ covering $\Y$.]{\label{fig4} \includegraphics[scale=0.85]{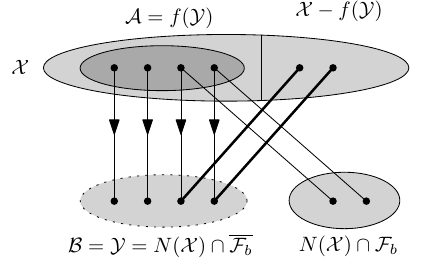}}
%\caption{}
\end{figure}

\textbf{Case 2:} $|\X|> |\Y|$. Suppose first that there exists a matching $f$ between $\Y$ and $\X$ covering $\Y$. Let $\A=f(\Y)$ and $\B=\Y$. Claim~\ref{claimfuncond} applied with $i=a-b$ on every pair $(f(C),C)\in (\A,\B)$, we have $e(\B, \cL_a)\le e(\A, \cL_b)$, so
%$$e(\B, \G_a)  = e(\B, \F_a) - e(\B, \A) \le e(\B, \cL_a) - e(\B, \A)\le e(\A, \cL_b) - e(\A, \B)  \le e(\A, \F_b). $$
$$e(\B, \G_a)+ e(\B, \A)  = e(\B, \F_a) \le e(\B, \cL_a) \le e(\A, \cL_b)  = e(\A, \F_b)+ e(\A, \B). $$
The inequality (\ref{eq_Levels_a_b}) follows by subtracting $e(\A, \B)$ from both sides. 

Suppose now that there is no matching covering $\Y$. By Hall's theorem, there must exist a minimal vertex set $\Y_0\subseteq \Y$ such that $|N(\Y_0)|<|\Y_0|$. Consider the following two subcases:
\begin{enumerate}[itemsep=2pt,parsep=2pt,topsep=2pt,partopsep=2pt]
\item[a)] There is a matching $f$ between $\Y_0$ and $N(\Y_0)$ covering $N(\Y_0)$. Let $\A=N(\Y_0)$ and $\B=f(N(\Y_0))$. There is no edge between $\B$ and $\G_a=\F_a \setminus \A$, hence $e(\B,\G_a)=0$ and the inequality (\ref{eq_Levels_a_b}) trivially holds.

\begin{figure}
\centering     %%% not \center
\subfigure[Case 2a.]{\centering \label{fig13}\includegraphics[scale=0.85]{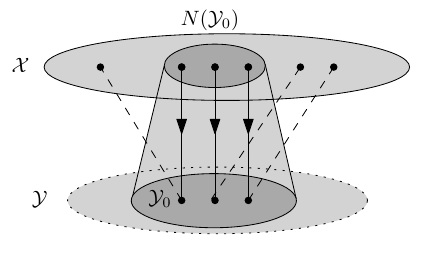}}
\hspace{1cm}
\subfigure[Case 2b.]{\centering \label{fig14}\includegraphics[scale=0.85]{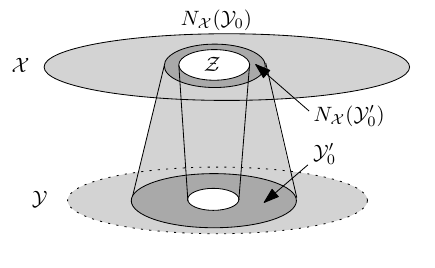}}
%\caption{my caption}
\end{figure}

%\begin{figure}
%\begin{center}
%\includegraphics{fig13.pdf}
%\includegraphics{fig14.pdf}
%\caption{•}
%\end{center}
%\end{figure}

\item[b)] There is no matching between $\Y_0$ and $N(\Y_0)$ covering $N(\Y_0)$. By Hall's theorem, there exists a vertex set $\Z\subseteq N(\Y_0)$ with $|N_{\Y_0}(\Z)|<|\Z|$. Then $\Y_0'\coloneqq \Y_0 \setminus N_{\Y_0}(\Z)$ is smaller than $\Y_0$. Since $|N_{\X}(\Y_0)|<|\Y_0|$ and $|\Z|>|N_{\Y_0}(\Z)|$, we also have
 $$|N(\Y_{0}')|\le |N(\Y_0)|- |\Z|<|\Y_0| - |N_{\Y_0}(\Z)|=|\Y_0'|,$$ and we can conclude that $\Y_0$ was not a minimal set with $|N(\Y_0)|< |\Y_0|$, a contradiction.
\end{enumerate}
We showed that there exists a $K$-optimal family $\F$ that is top-compressed.
The proof that $\F$ can ``be made'' bottom-compressed without increasing the number of comparable pairs follows by the above proof applied on $\F^c=\{X^c: X\in \F\}$.

 %If we apply the same proof on the family $\G^c$ of complements of the members of $\G$, we obtain an $M$-optimal family that is both top- and bottom-compressed.
\end{proof}
\end{proof}

Lemma~\ref{l_compressed} ensures the existence of an $M$-optimal top- and bottom-compressed family. Although we will use the lemma only for $M\le (1-\eps) \Sigma_{3}$, we emphasize that the result holds for any $M$, which might be of independent interest. Our next goal is to find an $M$-optimal family which additionally satisfies conditions (C1) and (C2) in the following definition. Recall that for $i\in\{0,1,2\}$ we have the notation $a_i=|\{j:A_j=i\}|$ and $b_i=|\{j:B_j=i\}|$.

\begin{definition}\normalfont
We say that a family $\F\subseteq \{0,1,2\}^n$ of size $M$ is $3$\emph{-compressed} if $\F$ is top-compressed, bottom-compressed, and additionally the following two conditions hold:
\begin{itemize}[itemsep=2pt,parsep=2pt,topsep=2pt,partopsep=2pt]
\item[(C1)] If $A$ is a maximum sized element of $\F$ with $|A|= n+2$ and $B\subseteq A$ is such that $|B|=n-1$ and $b_0>a_0$ then $B\in\F$.
\item[(C2)] If $A$ is a minimum sized element of $\F$ with $|A|= n-2$ and $B\supseteq A$ is such that $|B|=n+1$ and $b_2>a_2$ then $B\in\F$.
\end{itemize}
\end{definition}

The following claim is an analogue statement to Claim~\ref{claimfuncond}.
\begin{claim}\label{3compressclaim}
Let $A, B\in \{0,1,2\}^n$ such that $B\subseteq A$. If $|A|=n+2,|B|=n-1$, and $b_0 \neq a_0$, then for every $i\in\{1,2,3\}$, 
$$|N_{|B|+i}(B)|\le |N_{|A|-i}(A)|.$$
\end{claim}
\begin{proof}

Suppose that $|A|=n+2$ and $|B|=n-1$. Since $b_0\neq a_0$, we only need to consider the following two cases: 

\textbf{Case 1:} $b_2=a_2$.
The number of elements in levels $n+1$, $n$, and $n-1$, comparable with $A$, are 
$$\alpha_1\coloneqq a_2+a_1,\ \ \alpha_2\coloneqq  \binom{a_2+a_1}{2}+a_2,  \ \mbox{and} \ \alpha_3 \coloneqq  \binom{a_2+a_1}{3}+a_2\cdot (a_1+a_2-1), $$
respectively. Similarly, the number of elements in levels $n$, $n+1$, and $n+2$, comparable with $B$, is
$$\beta_1\coloneqq b_0+b_1,\ \mbox{and}\ \beta_2\coloneqq \binom{b_0+b_1}{2}+b_0, \ \mbox{and}\ \beta_3\coloneqq \binom{b_0+b_1}{3}+b_0\cdot (b_1+b_0-1),$$ 
respectively.
Note that $a_1=b_1+3$ and $a_2=b_2=b_0-1$, and so $a_2+a_1=b_0+b_1+2$.
We show that $\alpha_1\ge \beta_1$, $\alpha_2\ge \beta_2$, and $\alpha_3\ge \beta_3$. 
$$
\begin{array}{lllll}
\alpha_1- \beta_1&=&a_2+a_1- (b_0+b_1)= b_0+b_1+2-(b_0+b_1)> 0,\\
\alpha_2-\beta_2&=& \binom{b_0+b_1+2}{2}+(b_0-1)-\left(\binom{b_0+b_1}{2}+b_0\right)=2(b_0+b_1)\ge 0,\\
\alpha_3-\beta_3&=& \binom{b_0+b_1+2}{3}+(b_0-1)(b_1+b_0+1)-\left(\binom{b_0+b_1}{3}+b_0 (b_1+b_0-1)\right)\\
&=& b_0^2+2b_0 b_1 + b_1^2 +b_0 -b_1 -1.
\end{array}
$$
The last expression is negative only if $b_0=0$ and $b_1=1$, which is not possible since every element $B\in \cL_{n-1}$ must contain at least one $0$-coordinate.

\textbf{Case 2:} $b_2 \le a_2-1$ and $b_0\ge a_0+1$. Then
$$
b_0 = b_2+1 \le a_2\ \ \ \ \ \mbox{ and }\ \ \ \ \ b_2 =b_0-1\ge a_0.
$$
%$$
%\begin{array}{lllll}
%b_0 = b_2+1 \le a_2, \\
%b_2 =b_0-1\ge a_0,
%\end{array}
%$$
so $B^c$ has at most as many $2$'s and at least as many $0$'s as $A$, which implies that there exists a permutation $\pi(B^c)$ of the coordinates of $B^c$ such that $\pi(B^c)\subseteq A$.  This implies that for every $i\in \{1,2,3\}$,
$$|N_{|B|+i}(B)|=  |N_{|B^c|-i}(B^c)| = |N_{|B^c|-i}(\pi(B^c))| \le |N_{|A|-i}(A)|.$$

\end{proof}

\begin{lemma}\label{l_3-compressed}
For every natural number $K\le 3^n$, there exists a $K$-optimal family that is $3$-compressed. 
\end{lemma}

\begin{proof}
Let $\F$ be a $K$-optimal family in $\{0,1,2\}^n$ that is top- and bottom-compressed, whose existence is guaranteed by Lemma~\ref{l_compressed}. 
If $\F$ is not $3$-compressed, then at least one of the conditions (C1) and (C2) fails. We assume that (C1) does not hold, keeping in mind that in the other case we can apply the same proof on $\F^c$. Suppose that there exists a comparable pair $(A,B)$ in $\F$ such that $A$ is a maximal element with $|A|=n+2$, $|B|=n-1$, and $b_0>a_0$. To ease notation we write $a=|A|$ and $b=|B|$.  

Let $G$ be a bipartite graph with parts $\F_{a}$ and $\overline{\F_{b}}$ and with edges between comparable pairs $(A,B)$ for which $b_0\neq a_0$. As in the proof of Lemma~\ref{l_compressed}, we can iteratively replace some elements of $\F_{a}$ by elements of $\overline{\F_b}$ without increasing the number of comparable pairs. 
We need to consider several cases based on sizes of $\F_{a}$ and $\overline{\F_b}$ and existence of ``good'' matchings in $G$ that allow us to compress $\F$. 
Since $b<a$, the total value $\sum_{C\in \F} ||C|-n|$ of the family strictly decreases, ensuring that this process will terminate. These cases are the same as in the proof of Lemma~\ref{l_compressed}, except now we only consider matchings in the graph $G$ (in which all pairs with $b_0=a_0$ are removed), and we apply Claim~\ref{3compressclaim} at every place we applied Claim~\ref{claimfuncond} before.

%We show that if $A$ is a maximal element of $\F$ with $|A|= n+2$ and $B\subset A$ is such that $|B|=n-1$ and $b_0>a_0$ then $B\in\F$. 

%Because of Claim~\ref{3compressclaim} we can apply the same proof as the proof of Lemma~\ref{l_compressed}, but now on the bipartite graph $(\F_a,\overline{F_b})$ with an edge between $A\in \F_a$ and $B\in \F_b$ if and only if $B\subset A$ and $b_0\neq a_0$. 
\end{proof}

We are almost ready to tackle Theorem~\ref{mainresult}~(a). We will need to make use of the fact that a typical set in $\{0,1,2\}^n$ of size $n$ has about $n/3$ zeros, $n/3$ ones and $n/3$ twos.

\begin{claim}\label{averagethird}
For every $\eps>10 \left(\frac{1}{1.1}\right)^{0.005n}$,
$$\left|\left\{A\in\cL_{n+1}: \frac{0.9}{3}n \leq a_0 \leq \frac{1.1}{3}n \right\}\right|\geq (1-\eps^2) \ell_{n+1}.$$
\end{claim}
\begin{proof}
For every integer $c\geq 0$, let $f(c)\coloneqq |\{A\in\cL_{n+1} : a_0=c\}|$. Note that $f(c)=\binom{n}{c}\binom{n-c}{c+1},$ and hence 
$$\frac{f(c)}{f(c+1)}=\frac{(c+1)(c+2)}{(n-2c-2)(n-2c-1)}.$$
If $c>\frac{1.07}{3} n$ we get $f(c)/f(c+1)>1.1$ and if $c<\frac{0.93}{3} n=0.31n$ we have $f(c)/f(c-1)>1.1$. This means that 
$$\sum_{i\le 0.3n} f(i)\le \frac{1}{1-\frac{1}{1.1}}\cdot f(0.3n)\le 11\cdot \left(\frac{1}{1.1}\right)^{0.01 n}\cdot f(0.31n)\le \frac{\eps^2}{2}\cdot \ell_{n+1}.$$
 A similar computation gives $\sum_{i\ge 1.1n/3} f(i)\le \frac{\eps^2}{2}\cdot \ell_{n+1}$, and the claim follows.
\end{proof}

The next claim shows that for slightly varying values of $M$, the $M$-optimal families contain about the same number of comparable pairs. For an integer $N$, write $\comp(N)$ for the number of comparable pairs in an $N$-optimal family: 
$$\comp(N)\coloneqq \min\{\comp(\F): \F\subseteq \{0,1,2\}^n, |\F|=N\}.$$

\begin{claim}\label{continuouscomp}
If $M\le (1-\eps) \sum_{3}(n,2)$, then $\emph{comp}(M)\le  \emph{comp}(M-1)+\frac{n^2}{4}$.
\end{claim}
\begin{proof}%[Proof of Claim]
By the induction hypothesis, there exists an $(M-1)$-optimal centered family $\G$. Since $M\le (1-\eps) \sum_{3}$, the family $\G$ consists of all elements in layer $\cL_n$ and some elements in layers $\cL_{n-1}$ and $\cL_{n+1}$. 
%Define $\G_1\coloneqq \{B\in \cL_{n+1}:b_0\ge \frac{0.9}{3}n\}$ and $G_2\coloneqq \{B\in \cL_{n-1}: b_2\ge \frac{0.9}{3} n\}.$
Define $$\G_1\coloneqq \{B\in \cL_{n+1}:b_0\ge \frac{0.9}{3}n\}\ \ \mbox{ and }\ \ \G_2\coloneqq \{B\in \cL_{n-1}: b_2\ge \frac{0.9}{3} n\}.$$ 
Claim~\ref{averagethird} implies that $|\G_1|, |\G_2| \geq (1-\eps^2) |\cL_{n+1}|$. For $M\le (1-\eps) \Sigma_{3}$ we thus have $M< |\cL_{n}|+|\G_1|+|\G_2|$. Add an element $B\in (\G_1\cup \G_2)\setminus \G$ to $\G$. The element $B$ is in at most $\binom{2.1 n/3}{2}+n +1  \leq \frac{n^2}{4}$ comparable pairs of $\G\cup \{B\}$, hence $$\comp(M)\le \comp(\G\cup \{B\}) \le \comp(\G)+\frac{n^2}{4}= \comp(M-1)+\frac{n^2}{4}.$$
\end{proof}

We are ready to finish the proof of Theorem~\ref{mainresult}~(a). Let $\F$ be an $M$-optimal family that is $3$-compressed, whose existence is guaranteed by Lemma~\ref{l_3-compressed}, and assume that $\F$ is not centered. This can mean one of the following two things: 

\textbf{1.} The first possibility is that there exists an $A\notin \F$ of size $|A|=n$. Since $\F$ is both top- and bottom-compressed, this means that there is no $B\in\F$ with $A\subseteq B$ or $B\subseteq A$, hence unless $\F$ itself is an antichain we may decrease the number of comparable pairs in $\F$ by replacing one of its elements by $A$.

\textbf{2.} The second possibility is that $\cL_n\subseteq \F$ but $\F\not\subseteq \cL_{n-1}\cup \cL_n\cup \cL_{n+1}$. Then there exists an element $A\in\F$ of size at least $n+2$ or at most $n-2$. By symmetry we may assume that there exist $A\in\F$ with $|A|\geq n+2$. Let $A\in\F$ be a maximum element of $\F$. Since $\F$ is $3$-compressed, the number of elements in $\F_{a-1}\cup\F_{a-2}\cup\F_{a-3}$ comparable with $A$ is at least
\begin{equation}\label{number_nbrs}
\begin{array}{llllllllllll}
(a_1+a_2)&+&\left(\binom{a_1+a_2}{2}+a_2\right)&+&\left(\binom{a_1+a_2}{3}+a_2(a_1+a_2-1)-\binom{a_2}{3}\right).\\
\end{array}
\end{equation}
The term $\binom{a_2}{3}$ accounts for the elements of $\cL_{n-1}$ comparable with $A$ that have $a_0$ zeros, which are not necessarily in $\F$ by the definition of $3$-compressed (this case can only occur if $|A|=n+2$). Observe that every such element is formed by decreasing three $2$-coordinates of $A$ to $1$-coordinates, giving $\binom{a_2}{3}$ choices.
Since $a_1+a_2\ge n+2-a_2$, an elementary but somewhat tedious calculation shows that the quantity (\ref{number_nbrs}) is minimized when $a_1=0$ and $a_2=\frac{n+2}{2}$. It follows that this quantity is at least $a_2+\binom{a_2}{2}+a_2^2 \ge  \frac{3}{2}a_2^2 \ge \frac{3}{8} n^2.$ But then $\comp(\F)> \comp(\F \setminus \{A\})+\frac{3 n^2}{8}$, and $\F$ was not $M$-optimal (by Claim~\ref{continuouscomp}), a contradiction.

%\end{proof}
\hfill\qed

\section{Proof of Theorem~\ref{mainresult} (b)}

Recall that for an integer $a\geq 0$ and a family $\G\subseteq \{0,1,2\}^n$ we have the notation $\G_a=\{A\in\G:|A|=a\}$ and $\cL_a=\cL_a(n,2)$. We say that a centered family $\G\subseteq P$ is \emph{canonical centered} if there exists at most one $\ell\geq 0$ with $0<|\G_\ell|<|\cL_\ell(P)|$, i.e.~if it has at most one partial layer (while centered families could have two).  As in Section~\ref{sectionmaina}, whenever $A$ and $B$ are elements of $\{0,1,2\}^n$, we write $a_0,a_1,a_2$ and $b_0,b_1,b_2$ for the number of $0$-, $1$-, $2$-coordinates in $A$ and $B$ respectively. For an element $A\in\{0,1,2\}^n$ and family $\G\subseteq \{0,1,2\}^n$, we use the notation $\comp(A,\G)\coloneqq |\{B\in\G:B\subsetneq A \text{ or } A \subsetneq B\}|$ and $\Comp(\G)\coloneqq\{(A,B)\in\G\times \G : A\subset B\}$, so that $|\Comp(\G)|=\comp(\G)$.

 Let $X=(0,0,1,1,\ldots,1)\in\cL_{n-2}$, $\B=\{B\in\cL_{n+3}: b_0=0\}$, and $\C\coloneqq \{C\in \{0,1,2\}^n: n-2\leq|C|\leq n+3\}$. Finally, let $\F\coloneqq \C \setminus \left(\B\cup \{X\}\right)$ (see Figure \ref{fig_counterexample1}). Then $\F$ is not centered, but we claim that $\F$ contains fewer comparable pairs than every centered family of size $M=|\F|=\Sigma_6(n,2)-\binom{n}{3}-1$. 
The proof of this claim goes in two stages. First we show that $\F$ contains fewer comparable pairs than the best canonical centered family of this size (Claim~\ref{cl_F}), and next we show that among centered families of this size the canonical families are the best (Lemma~\ref{bestcenterediscanon}). 

\begin{figure}
\begin{center}
\includegraphics{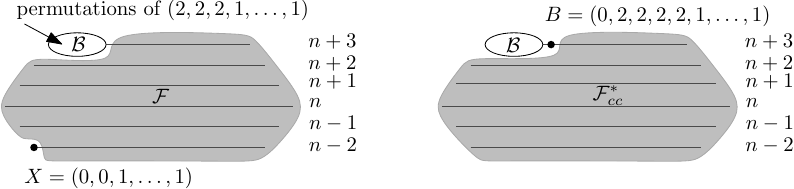}
\caption{A non-centered family $\F$ and a canonical centered family $\F_{cc}^*$ such that $\comp(\F)<\comp(\F_{cc}^*)\le \comp(F_{cc})$ for every canonical centered family $\F_{cc}$ of size $M=\Sigma_6-\binom{n}{3}-1$.}\label{fig_counterexample1}
\end{center}
\end{figure}

\begin{claim}\label{cl_F}
Whenever $\F_{cc}$ is a canonical centered family of size $M=\Sigma_6(n,2)-\binom{n}{3}-1$ we have $\emph{comp}(\F)<\emph{comp}(\F_{cc})$.
\end{claim}
\begin{proof}

Every canonical centered family $\F_{cc}$ of size $M=\Sigma_6(n,2)-\binom{n}{3}-1$ consists of all elements in levels $\cL_{n-2},\dots, \cL_{n+2}$ and $\ell_{n+3}-\binom{n}{3}-1$ elements in $\cL_{n+3}$ (or $\ell_{n-3}-\binom{n}{3}-1$ elements in $\cL_{n-3}$, in which case the proof is symmetrical). Let $B=(0,2,2,2,2,1,1,\dots,1)\in \cL_{n+3}$ and note that $\F_{cc}^*:=\F\cup \{X\} \setminus \{B\}$ is one of the canonical centered families of size $M$ with the least number of contained comparable pairs. Indeed, removing all elements with no $0$-coordinates plus one element with one $0$-coordinate from $\cL_{n+3}$ ensures the smallest possible number of comparable pairs. This can be seen because it is always better to replace a $2$-coordinate and a $0$-coordinate by two $1$-coordinates, or directly from the formula (\ref{number_nbrs}).

It suffices to show that $\comp(B,\F)<\comp(X,\F)$ since then we can improve $\F\cup \{X\} \setminus \{B\}$ by deleting $X$ and adding $B$. Now, $\comp(X,\F)\ge \binom{n}{5}+\binom{n}{4}-O(n^3)$ whereas $\comp(B,\F)=\binom{n-1}{5}+\binom{n-1}{4}+O(n^3)$, which is $\Theta(n^4)$ smaller than $\comp(X,\F)$ and the claim follows.
\end{proof}

\begin{lemma}\label{bestcenterediscanon}
%Let $M=\Sigma_6(n,2)-\binom{n}{3}-1$. For every centered family $G$ of size $M$ there exists a canonical centered family $\F_{cc}$ of size $M$ such that $\comp(\F_{cc})\le \comp(\G)$.
%Let $\F_c$ be a centered family of size $M$. Then $\emph{comp}(\F_c)\geq \emph{comp}(\F\cup X \setminus B)$. That is, 
Among centered families of size $M=\Sigma_6(n,2)-\binom{n}{3}-1$ the function $\emph{comp}(\cdot)$ attains its minimum on a canonical centered family.
\end{lemma}
\begin{proof}
Define a partial order on the collection of centered families of size $M$ by letting $\HH <\HH'$ if $\comp(\HH)<\comp(\HH')$, or if $\comp(\HH)=\comp(\HH')$ and $|\HH_{n+3}| > |\HH'_{n+3}|$. We will show that one of the minimal elements of this partial order is canonical centered, which immediately implies Lemma~\ref{bestcenterediscanon}. 
Let $\G$ be a centered family of size $M=\Sigma_6(n,2)-\binom{n}{3}-1$ that is minimal according to this ordering. Note that $\cL_{n-2}\cup \dots \cup \cL_{n+2} \subseteq \G \subseteq \cL_{n-3}\cup \dots \cup \cL_{n+3}$.

Given a permutation $\pi\in S_n$ of order $2$ (i.e., $\pi^2=1$) define the $\pi$-\emph{compression} of $\G$ by ``replace $A\in \G_{n-3}$ by $\pi(A^c)$ unless it is already in $\G_{n+3}$''. That is,
$$\picomp(\G)=\G \cup \{\pi(A^c) \in \cL_{n+3}: A \in\G_{n-3}\}  \setminus \{A\in \G_{n-3}: \pi(A^c)\not \in\G_{n+3}\}.$$

\begin{claim}\label{claim_compressionbest}
For every $\pi\in S_n$ of order $2$ we have $\picomp(\G)<\G$, unless $\G=\picomp(\G)$. That is, $\pi$-compression improves the family unless it is already $\pi$-compressed.
\end{claim}
\begin{proof}
Note first that unless $\G=\picomp(\G)$ we have that $\picomp(\G)_{n+3}>\G_{n+3}$. It thus remains to show $\comp(\picomp(\G))\le \comp(\G)$. Suppose that $B\subset \pi(A^c)$ is a new comparable pair. Then $A$ was replaced by $\pi(A^c)$, so $A \in \G \setminus \picomp(\G)$. 
The element $B$ was not replaced by $\pi(B^c)$, so $\pi(B^c)\in \G$. Observe that for every $\pi\in\S_n$,  $B\subset \pi(A^c)$ implies $A\subset \pi^{-1}(B^c)$. Since our permutation $\pi$ is of order $2$, we have $\pi^{-1}(B^c)=\pi(B^c)$, and thus $A\subset \pi(B^c)$. Together, for every new comparable pair $B\subset\pi(A^c)$ there is an old comparable pair $A\subset \pi(B^c)$ which got deleted during the compression. This defines an injection from $\Comp(\picomp(\G))\setminus \Comp(\G)$ into $\Comp(\G) \setminus \Comp(\picomp(\G))$ and the claim follows.
\end{proof}

We sketch the idea of the remaining part of the proof. By Claim~\ref{claim_compressionbest} and the minimality of $\G$, the family $\G$ is $\pi$-compressed for all permutations $\pi$ of order $2$. For $A\in \G_{n-3}$, define $$\Pi(A^c)\coloneqq \{\pi(A^c)\in \cL_{n+3}:\pi\in S_n \text{ and } \pi^2=1\},$$ and count the  elements of $\Pi(A^c)$ comparable with $A$. Every such element has to be in $\G_{n+3}$ by definition of $\pi$-compression. To obtain a superset of $A$ in $\Pi(A^c)$, we first need to switch all $0$-coordinates of $A^c$ with some of its $2$-coordinates. After that we can freely switch any of the remaining three $2$-coordinates with any three $1$-coordinates. Any permutation that is formed in this fashion is obviously of order $2$. The number of such permutations is $\binom{a_0}{3} \binom{a_1}{3}$. It follows that if the number of $0$'s and $1$'s in $A$ is (close to) linear in $n$, then the number of elements in $\G_{n+3}$ comparable with $A$ is of order (close to) $n^6$. Therefore, $\G_{n-3}$ cannot have many such elements since otherwise we could replace $\G_{n-3}$ by elements of $\cL_{n+3}\setminus \G$ and the number of comparable pairs would decrease. We partition $\G$ into $\G'$, $\G''$, and $\G^*$ as follows:
%\mathclose{} is used below to get rid off big space after \right}
$$\G' = \left\{A\in \G_{n-3} : a_2\leq n^{2/3}\log n\right\}\mathclose{},\ \G'' = \left\{A\in\G_{n-3}: a_2\geq \frac{n}{2}-n^{2/3}\log n\right\}\mathclose{}, \ \G^* = \G_{n-3} - \left(\G' \cup \G''\right)\mathclose{}.$$
%Since $a_0=a_2+3$, 
Observe that $\G'$ contains elements with a small number of $0$- and $2$-coordinates while $\G''$ contains elements with small number of $1$-coordinates. Claim~\ref{gstarsmall} states that there cannot be more elements in $\G^*$ than in $\G'\cup \G''$. Claim~\ref{cl_GH} uses a similar averaging argument to bound 
$|\G'\cup \G''|$ by $2|\HH'\cup \HH''|$, where $\HH' \cup \HH''$ is the family of sets in $\G'\cup \G''$ that are in a small number of comparable pairs in $\G$.  Claim~\ref{cl_Hbounds} then implies that $\HH' \cup \HH''$ must be empty, and we conclude that $\G$ is canonical centered.

\begin{claim}\label{gstarsmall}
$|\G_{n-3}|\leq 2|\G'\cup \G''|.$
\end{claim}
\begin{proof}
Note that this claim is equivalent to the inequality $|\G^*|\leq|\G'\cup\G''|$. Let $A$ be an element of $\G^*$ and consider all its supersets of the form $\pi(A^c)$ with $\pi^2=1$. Since $\G$ is $\pi$-compressed for every involution $\pi$, we know that all these supersets are in $\G$. Let $\Pi_A $ be the set of a permutations $\pi$ of order $2$ such that each $\pi$ switches all $0$-coordinates of $A^c$ with all but three of its $2$-coordinates, and the remaining three $2$-coordinates with three arbitrary $1$-coordinates. Equivalently, for every $\pi\in \Pi_A$, the element $\pi(A^c)$ is formed from $A$ by increasing three $0$-coordinates and three $1$-coordinates by one. We thus always have $A\subset \pi(A^c)$, and hence the number of supersets of $A$ in $\Pi(A^c)$ is at least $|\Pi_A|=\binom{a_0}{3}\binom{a_1}{3}$. Since $A\not \in G'\cup\G''$ and $a_0=a_2+3$, we have
$$n^{2/3} \log n +3 \le a_0\le \frac{n}{2} - n^{2/3} \log n+3.$$
From $a_0+a_1+a_2=n$ we have $a_1=n-2a_0+3$, and thus
$$a_1\ge n-2\left(\frac{n}{2}-n^{2/3} \log n+3\right)+3=2n^{2/3} \log n -3.$$
As either $a_0$ or $a_1$ is larger than $n/10$, we have
$$|\Pi_A|=\binom{a_0}{3}\binom{a_1}{3}\geq n^5\log^2 n.$$
We claim that the elements of $\G_{n-3}$ are in at most $n^5$ comparable pairs each on average. Indeed, otherwise we could replace $\G_{n-3}$ by an arbitrary subset of $\overline{\G_{n+3}}=\cL_{n+3}\setminus \G_{n+3}$ of size $|\G_{n-3}|$ and obtain a canonical centered family with a smaller number of comparable pairs. Because each element of $\G^*$ is in at least $n^5 \log^2 n$ comparable pairs, we have $|\G^*|\le |\G'\cup \G''|$, and the claim follows.
\end{proof}
Let $$\HH' = \left\{A\in\G' : \comp(A,\G)\leq 2n^5\right\} \ \ \ \mbox{ and }\ \ \ \HH'' = \left\{A\in\G'' : \comp(A,\G)\leq 2n^5\right\}\mathclose{}. $$ 

\begin{claim}\label{cl_GH}
$|\G' \cup \G''|\leq 2|\HH'\cup \HH''|.$
\end{claim}
\begin{proof}
As before, the elements of $\G_{n-3}$ must be in at most $n^5$ comparable pairs each on average since otherwise we could replace $\G_{n-3}$ by an arbitrary subset of $\overline{\G_{n+3}}$. Recall that the family $\G_{n-3}$ is partitioned into $\G'$, $\G''$ and $\G^*$, and that every element of $\G^*$ is in at least $n^5 \log^2 n$ comparable pairs of $\G$ (see proof of Claim~\ref{gstarsmall}). We thus necessarily have $|\G'\cup \G''|\le 2 |\HH' \cup \HH''|$.
\end{proof}

\begin{claim}\label{cl_Hbounds}
$$|\HH'|,|\HH''| \leq  \frac{\log^8n}{n^2} \cdot |\overline{\G_{n+3}}|.$$
\end{claim}

\begin{proof}
We first count the number $E''$ of comparable pairs $(A,B)\in \HH'' \times \overline{\G_{n+3}}$ such that $a_2=b_2$. We count $E''$ two ways:
\begin{enumerate}[itemsep=2pt,parsep=2pt,topsep=2pt,partopsep=2pt]
\item Let $A\in\HH''\subseteq \G''$. Then $a_0=a_2+3\ge \frac{n}{2}-n^{2/3}\log n+3$ by the definition of $\G''$. We need to count the number of 
sets $B\in \overline{\G_{n+3}}$ formed from $A$ by increasing six of its $0$-coordinates to $1$-coordinates. Since $\comp(A,\G_{n+3})\le 2n^5$ by the definition of $\HH''$, this number is at least 
$\binom{a_0}{6}-2n^5 \ge \binom{n/3}{6}\geq n^6/10^{9}.$

\item Let now $B\in \overline{\G_{n+3}}$ for which there exists an $A\in \HH''$ with $a_2=b_2$. Then 
$$b_1=n+3-2b_2=n+3-2a_2\le n+3-2\left(\frac{n}{2}-n^{2/3}\log n\right)\le 3n^{2/3}\log n.$$
Therefore, the number of sets $A$ formed from $B$ by decreasing six of its $1$-coordinates to $0$-coordinates is 
at most $\binom{3n^{2/3}\log n}{6}\leq n^4\log^7n.$
\end{enumerate}
Together we obtain 
\begin{equation}\label{eq_H''} 
|\HH''|\cdot \frac{n^6}{10^{9}}\leq E'' \leq |\overline{\G_{n+3}}| \cdot  n^4\log^7n,
\end{equation}
and the second inequality in Claim~\ref{cl_Hbounds} follows.

Similarly, we count the number $E'$ of comparable pairs $(A,B)\in \HH' \times \overline{\G_{n+3}}$ such that $a_0=b_0$.
\begin{enumerate}[itemsep=2pt,parsep=2pt,topsep=2pt,partopsep=2pt]
\item Let $A\in\HH'\subseteq \G'$. Then $a_1=n-3-2a_2\ge n-3-2n^{2/3} \log n$ by the definition of $\G'$. The number of 
sets $B\in \overline{\G_{n+3}}$ formed from $A$ by increasing six of its $1$-coordinates to $2$-coordinates is at least 
$\binom{a_1}{6}-2n^5 \ge \binom{n/3}{6}\geq n^6/10^{9}.$

\item Let now $B\in \overline{\G_{n+3}}$ for which there exists an $A\in \HH'$ with $a_0=b_0$. Then 
$b_2=a_2+6\le n^{2/3}\log n+6.$
Therefore, the number of sets $A$ formed from $B$ by decreasing six of its $1$-coordinates to $0$-coordinates is 
at most $\binom{n^{2/3}\log n+6}{6}\leq n^4\log^7n.$
\end{enumerate} 
Similarly to \ref{eq_H''} we have $$|\HH'|\cdot \frac{n^6}{10^{9}}\leq E' \leq |\overline{\G_{n+3}}| \cdot  n^4\log^7n,$$
and the first inequality in Claim~\ref{cl_Hbounds} follows.
%We obtain bounds on $E'$ analogous to the bounds on $E''$ in \eqref{eq_H''}, and the first inequality in Claim~\ref{cl_Hbounds} follows.
\end{proof}
 
We are ready to finish the proof of Lemma~\ref{bestcenterediscanon}. Applying the previous three claims, we obtain
$$|\G_{n-3}|\stackrel{C\ref{gstarsmall}}{\le} 2 |\G' \cup \G''|\stackrel{C\ref{cl_GH}}{\leq} 4 |\HH' \cup \HH''| \stackrel{C\ref{cl_Hbounds}}{\leq} \frac{\log^9n}{n^2}\cdot  |\overline{\G_{n+3}}|= \frac{\log^9n}{n^2}\left(|\G_{n-3}|+\binom{n}{3}+1\right)\mathclose{},$$
and therefore
\begin{equation}\label{eq_contra}
|\G_{n-3}|\leq n\log^{10}n.
\end{equation}
Assume that $\HH''\neq 0$ and let $A\in \HH''$. As in the proof of Claim~\ref{cl_Hbounds}, $\comp(A,\overline{\G_{n+3}}) \ge \binom{a_0}{6}-2n^5 \geq n^6/10^{9},$ and so $\overline{|\G_{n+3}}|\ge n^6/10^{9}$.
%By the definition of $\HH''$ we have $|\overline{\G_{n+3}}|\geq n^6/10^{9}$. 
This implies $|\G_{n-3}|\geq n^6/10^{10}$, which contradicts equation \eqref{eq_contra}. By the same argument we have $\HH'=\emptyset$. Hence $\G_{n-3}=\emptyset$ by Claims~\ref{gstarsmall} and \ref{cl_GH}, 
and we conclude that $\G$ is canonical centered, proving the lemma.
\end{proof}

\section{Proof of Theorem~\ref{largeresult}} 

Let $P=\{0,1,\ldots,k\}^n$ where $k$ is a fixed constant, $0<\eps<0.01$, and $n$ be sufficiently large so that all following estimates hold. We are given an integer $j$ with $(1+\eps)\log_2 n \leq j\leq \sqrt{n}/\log_2n$ and we have $M=\Sigma_j(n,k)$. For simplicity we will assume $nk+j$ is even, the odd case is very similar, and we omit the details. 
Let $$\F\coloneqq \left\{A\in P : \frac{nk-j}{2}< |A|\leq \frac{nk+j}{2}\right\}\mathclose{}.$$ 
Let $B$ be such that $|B|=\frac{nk+j}{2}$ and every coordinate of $B$ is either $\lfloor \frac{k}{2}\rfloor$ or $\lfloor \frac{k}{2}+1\rfloor$. Let $C$ be such that $|C|=\frac{nk+j}{2}+1$ and every coordinate of $C$ is $k$ or $0$, except possibly one. Note that $C$ has at least $\frac{n-j}{2}$ zeros and at most $\frac{n+j}{2}$ non-zeros.
\begin{figure}
\begin{center}
\includegraphics{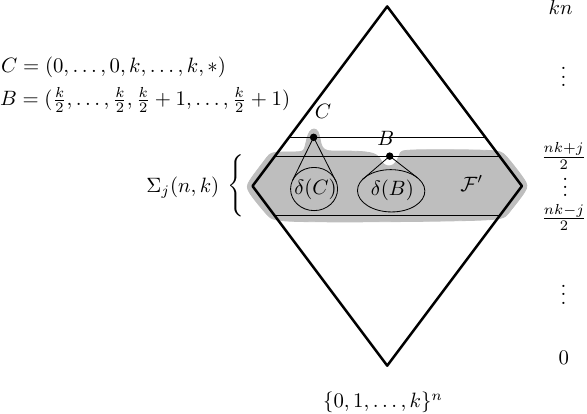}
\caption{A non-centered family $\F'\subseteq \{0,\dots, k\}^n$ which has smaller number of comparable pairs than the centered family $\F$.}\label{fig_poset6}
\end{center}
\end{figure}
 Now define $$\F'\coloneqq \F \cup \{C\} \setminus \{B\},$$ so that $\F'$ is not a centered family (see Figure~\ref{fig_poset6}). We claim that $\comp(\F')<\comp(\F)$. We only need to compare the number of subsets of $B$ and $C$ that are contained in $\F$ (or $\F'$). For a set $D$ and an integer $\ell$, write $$\delta_\ell(D)\coloneqq \{A\in\F: A\subseteq D,~  |A| = |D| - \ell\},$$ that is, the collection of subsets of $D$ that are in $\F$, and are $\ell$ levels below $D$. Let $$\delta(D)=\bigcup_{\ell=0}^n \delta_\ell(D).$$ We have the estimate
$$|\delta(B)|=\sum_{\ell=0}^{j-1}|\delta_\ell(B)|>|\delta_{j-1}(B)|>\binom{n}{j-1}\mathclose{}.$$
Note that for $0\leq \ell\leq j$ we have $$|\delta_\ell(C)|\leq \binom{\frac{n+j}{2}+\ell - 1}{\ell} = (1+o(1))\binom{\frac{n+j}{2}}{\ell},$$ since the right hand side of the first inequality counts the number of non-negative solutions to the equation $a_1+\ldots + a_{(n+j)/2} = \ell$. Hence we get
\begin{equation*}
\begin{split}
|\delta(C)|&=\sum_{\ell=0}^{j}|\delta_\ell(C)|\leq(1+o(1))\sum_{\ell=0}^{j}\binom{\frac{n+j}{2}}{\ell} \leq 2\binom{(0.5+\eps^{3/2})n}{j}\\ & \leq n\cdot (0.5+\eps^{4/3})^{(1+\eps)\log_2 n}\binom{n}{j-1}<|\delta(B)|,
\end{split}
\end{equation*}
where the last inequality holds because $(0.5+\eps^{4/3})^{1+\eps}<\frac{1}{2}$ for $\eps<0.01$.
Hence $\comp(\F')<\comp(\F)$ and this completes the proof.
$~\hfill \qed$

\section{Proof of Theorem~\ref{subspacebs}}
%Let $P$ be a rank-symmetric and rank-unimodal poset of rank $n$.
Recall that $\RSU(n)$ denotes the collection of posets of rank $n$ that are rank-symmetric and rank-unimodal, and let $P\in\RSU(n)$.
 Furthermore, recall that $|A|$ denotes the rank of an element $A\in P$, $\comp(A,\G)\coloneqq |\{B\in\G:B\subset A \text{ or } A \subset B\}|$, and $N_r(A)\coloneqq \{B: |B|=r, B\subseteq A \text{ or } A\subseteq B\}$. 

A poset $P$ of rank $n$ has \emph{property} $(Q)$ if all of the following hold:
\begin{itemize}[itemsep=2pt,parsep=2pt,topsep=2pt,partopsep=2pt]
\item[(Q1)] If $|B|<|A|$ and $||B|-n/2|<||A|-n/2|$, then $|N_{|B|+i}(B)|\leq |N_{|A|-i}(A)|$ for every $i\in \{1,\ldots, |A|-|B|\}$.
\item[(Q2)] If $|B|>|A|$ and $||B|-n/2|<||A|-n/2|$, then $|N_{|B|-i}(B)|\leq |N_{|A|+i}(A)|$ for every $i\in \{1,\ldots, |B|-|A|\}$.
\item[(Q3)] If $n/2\leq |B|<|A|$, then $|N_{|B|-i}(B)|\leq |N_{|A|-i}(A)|$ for every $i\geq 1$.
\item[(Q4)] If $n/2\geq |B|>|A|$, then $|N_{|B|+i}(B)|\leq |N_{|A|+i}(A)|$ for every $i\geq 1$.
\end{itemize}

%\begin{definition}
%We say that $P$ has \emph{Property (Q)} if for every $(A,B)\in \Comp(P)$ all of the following hold (where $a=|A|$ and $b=|B|$):
%\begin{itemize}
%\item[(Q1)] If $b<a$ and $|b-n/2|<|a-n/2|$, then $|N_{b+i}(B)|\leq |N_{a-i}(A)|$ for every $i\in [a-b]\}$.
%\item[(Q2)] If $b>a$ and $|b-n/2|<|a-n/2|$, then $|N_{b-i}(B)|\leq |N_{a+i}(A)|$ for every $i\in [b-a]\}$.
%\item[(Q3)] If $n/2\leq b<a$, then $|N_{b-i}(B)|\leq |N_{a-i}(A)|$ for every $i\geq 1$.
%\item[(Q4)] If $n/2\geq b>a$, then $|N_{b+i}(B)|\leq |N_{a+i}(A)|$ for every $i\geq 1$.
%\end{itemize}
%\end{definition}

The key result of this section is the lemma below, which will easily imply Theorem~\ref{subspacebs}.

\begin{lemma}\label{kleitmanstatedgenerally}
%For every poset $P\in \RSU$, if $P$ has Property (Q), then $P$ has the centeredness property.
If a rank-symmetric and rank-unimodal poset $P$ of rank $n$ has Property (Q), then $P$ has the centeredness property.
\end{lemma}

\begin{proof}%[Proof of Lemma~\ref{kleitmanstatedgenerally}]
Suppose $P\in\RSU(n)$  has Property (Q). We say that a family $\F\subseteq P$ is \emph{mid-compressed} if for every comparable pair $(A,B)\in \Comp(\F)$ such that $||B|-n/2|<||A|-n/2|$, $A\in\F$ implies $B\in \F$.

\begin{claim}\label{l_compressed_gen}
For every $M\in \{1,\dots, |P|\}$, there exists an $M$-optimal family in $P$ that is mid-compressed.
\end{claim}

\begin{proof}
The proof of this claim is essentially the same as Kleitman's proof~\cite{kleitman} of Theorem~\ref{kleitmanthm} and hence similar to our proof of Lemma~\ref{l_compressed}, so we only give a sketch here. We show by induction on $M$ that there exists an $M$-optimal family that is centered. The base case is $M\le \Sigma_{1}(n,k)$, in which case there exists an antichain in $\cL_{n/2}$ of size $M$. 

Now let  $M> \Sigma_{1}(n,k)$, and define an order relation on the collection of subsets of $P$ of order $M$ by setting $\G < \F$ if

\begin{itemize}[itemsep=2pt,parsep=2pt,topsep=2pt,partopsep=2pt]
\item $\comp(\G)<\comp(\F)$, or
\item $\comp(\G)=\comp(\F)$ and $\sum_{G\in \G}||G|-n/2|<\sum_{F\in\F}||F|-n/2|$.
\end{itemize}
Given a family  $\F\subset P$ of size $M$ that is not mid-compressed we will find a family $\G$ of size $M$ that improves $\F$ (that is, $\G<\F$). Since only mid-compressed families cannot be improved this way this will show that there exists an $M$-optimal mid-compressed family. 
%Showing that the best mid-compressed families are centered will then finish the proof.

Let $\F\subset P$ be a family of size $M$ that is not mid-compressed. Then there exist elements $A$ and $B$ such that $A\in \F$, $B\notin \F$, and $||B|-n/2|<||A|-n/2|$. Without loss of generality we may assume that there exists such a pair with $|A|>n/2$. Among all such pairs $(A,B)$ consider the pairs with $|A|$ is maximal, and then among these pick one with $|B|$ maximal. Note that this implies that whenever $C\in P$ is such that $C\subset A$ and $|C|>|B|$ then $C\in\F$. Moreover whenever $C\in P$ is such that $B\subset C$ and $|C|>|A|$ then $C\notin \F$. Let $a\coloneqq |A|$  and $b\coloneqq |B|$.

Form a bipartite graph with vertex sets $\F_{a}$ and $\overline{\F_{b}}$ with edges between comparable pairs. If there exists a matching $f$ between $\F_{a}$ and $\overline{\F_{b}}$ covering $\F_{a}$, then replacing $\F_{a}$ with the matching elements $f(\F_{a})$ does not increase the number comparable pairs in $\F$ (since $P$ has Property (Q1)), but decreases $\sum_{F\in\F}||F|-n/2|$ and hence improves the family. From now on suppose that there is no such matching. Let $\X=\F_{a}$ and let $\Y$ be the family of neighbors of $\F_{a}$ in $\overline{\F_{b}}$. 

\textbf{Case 1:} $|\X|\leq |\Y|$. Since there is no matching between $\X$ and $\Y$ covering $\X$, we can find a maximal vertex set $\X_0\subset \X$ such that $|N(\X_0)|<|\X_0|$. Let $f$ be a matching between $\F_{a}-\X_0$ and $\Y - N(\X_0)$ covering $\X- \X_0$, which exists by the maximality of $\X_0$. Then $\G\coloneqq \F\cup f(\X-\X_0)- (\X- \X_0)$ satisfies $\G<\F$ (again using that $P$ has Property (Q1)).

\textbf{Case 2:} $|\X|>|\Y|$. If there exists a matching $f$ covering $\Y$ then replacing $f(\Y)$ by $\Y$ improves $\F$. Otherwise, let $\Y_0\subset \Y$ be minimal such that $|N(\Y_0)|<|\Y_0|$. Consider the following two cases:

\begin{enumerate}[itemsep=2pt,parsep=2pt,topsep=2pt,partopsep=2pt]
\item[a)] If there is a matching $f$ between $\Y_0$ and $N_{\X}(\Y_0)$ covering $N(\Y_0)$, then let $\G\coloneqq(\F\setminus N_{\X}(\Y_0)) \cup f(N_{\X}(\Y_0))$.
 Since there is no edge between $f(N_{\X}(\Y_0))$ and $\F_a$, we have $\comp(\G)<\comp(\F)$.

\item[b)] Otherwise, there exists a vertex set $\Z\subseteq N_{\X}(\Y_0)$ with $|N_{\Y_0}(\Z)|<|\Z|$. Then $\Y_0'\coloneqq \Y_0 \setminus N_{\Y_0}(\Z)$ is smaller than $\Y_0$ and it is easy to check that $|N_{\X}(\Y_0)|< |\Y_0|$, a contradiction with minimality of $\Y_0$.
\end{enumerate}
%Then as in the proof of Lemma~\ref{l_compressed} we can find a matching $f $ covering $N(\Y_0)$, and then $\G\coloneqq \F- N(\Y_0)\cup f(N(\Y_0))$ is an improvement over $\F$. 
This finishes the proof of the claim that there exists an $M$-optimal mid-compressed family. 
\end{proof}
%%%%%%%%%%%%%%%%%%%%%%%%%%%%%%%%%%%%%%%%%%%%%%%%%%%%%%%%%%%%%%%%%%%%%%%%%%%%%%%%%%%%%%%%%%%%%%%

From now on we assume that there exists an $M$-optimal mid-compressed family $\F^*$ that is not centered. Recall that $\Sigma_r(P)$ denotes the total size of the middle $r$ layers of $P$. Define the integer $j\geq 0$ such that $\Sigma_{j-1}(P)<M\leq \Sigma_j(P)$. 
Let $\G\subset P$ be the centered family of size $\Sigma_j(P)$ and write $\Delta(\G)\coloneqq \max\{\comp(A,\G):A\in \G\}$ for the maximum degree of the graph with vertex set $\G$ and edges corresponding to comparable pairs in $P$. Let $\comp(M-1)\coloneqq \min\{\comp(\F):\F\subseteq P, |\F|=M-1\}$.
The following statement is very similar to Claim~\ref{continuouscomp}:
\begin{claim}\label{lastcontinuousclaim}
We have $\emph{comp}(\F^*)\leq \emph{comp}(M-1)+\Delta(\G)$.
\end{claim}
\begin{proof}
It suffices to construct a family $\F$ of size $M$ with at most $\comp(M-1)+\Delta(\G)$ comparable pairs. As $\F^*$ is $M$-optimal it contains at most this many comparable pairs. By induction we know there exists a centered $(M-1)$-optimal family $\HH$. Since $\HH\subset \G$, adding to it any element of $\G\setminus \HH$ increases the number of comparable pairs by at most $\Delta(\G)$.
\end{proof}

Since $\F^*$ is not centered, it contains an element $A$ such that for all elements $B\in\G$ we have $||A|-n/2|>||B|-n/2|$. Since $\F^*$ is mid-compressed and $P$ has properties (Q3) and (Q4), this implies that $\comp(A,\F^*)\geq\Delta(\G)$. Hence $\comp(\F^*)\geq \comp(M-1)+\Delta(\G)$. By Claim~\ref{lastcontinuousclaim} this implies that every family of size $M$ contains at least $\comp(M-1)+\Delta(\G)$ comparable pairs. As shown in the proof of Claim~\ref{lastcontinuousclaim} this value can be achieved by a centered family, completing the proof of Lemma~\ref{kleitmanstatedgenerally}.
\end{proof}

One well-known poset that satisfies the assumptions of Lemma~\ref{kleitmanstatedgenerally} is the Boolean lattice $\P(n)$. Therefore, Lemma~\ref{kleitmanstatedgenerally} implies Theorem~\ref{kleitmanthm}---rather unsurprisingly since the proof of Lemma~\ref{kleitmanstatedgenerally} was motivated by Kleitman's proof of Theorem~\ref{kleitmanthm}. 

Let $q$ be a prime power and let $n\ge 1$. To finish the proof of Theorem~\ref{subspacebs}, we only need to check that the assumptions of Lemma~\ref{kleitmanstatedgenerally} hold for $\mathcal{V}(q,n)$.

\begin{claim}\label{vqnranksym}
$\mathcal{V}(q,n)$ is rank-symmetric.
\end{claim}
\begin{proof}
The map $V\mapsto V^\perp$ takes the set of subspaces of dimension $k$ into the set of subspaces of dimension $n-k$ bijectively. 
\end{proof}

\begin{claim}
$\mathcal{V}(q,n)$ is rank-unimodal.
\end{claim}
\begin{proof}
Note that the number of subspaces of $\mathcal{V}(q,n)$ of dimension $k$, written as ${n\brack k}_q$, can be expressed as (see e.g.~\cite{stanley}):
$${n\brack k}_q=\frac{[n]!}{[k]![n-k]!},$$
where 
$$[n]!=[1]\cdot[2] \cdot \ldots \cdot [n],\quad \text{and ~} [i]=q^i-1.$$
Rank-unimodality of $\mathcal{V}(q,n)$ is easily seen to follow from this formula.
\end{proof}

\begin{claim}\label{vqnpropq}
$\mathcal{V}(q,n)$ has Property (Q).
\end{claim}
\begin{proof}
 Properties (Q1)--(Q4) follow from the observation that if $S$ is a subspace of $\mathbb{F}_q^n$ of dimension $m$ then the number of spaces $S'\subset S$ of dimension $m-k$ is ${m \brack k}_q$ and the number of spaces $S'$ with $S\subset S'$ and $\dim(S')=m+k$ is ${n-m\brack k}_q$. 
\end{proof}

The proof of Theorem~\ref{subspacebs} now follows from putting together Lemma~\ref{kleitmanstatedgenerally} and Claims~\ref{vqnranksym}--\ref{vqnpropq}.

\section{Open problems}

Recall that $\RSU$ is the collection of posets that are rank-symmetric and rank-unimodal and let $\bC\subset \RSU$ be the collection of posets which have the centeredness property. The main open problem that this paper has only barely begun to explore asks for an easy way to decide whether a poset $P\in\RSU$ is in $\bC$. We know that $\{0,1\}^n\in\bC$ and $\mathcal{V}(q,n)\in\bC$ but for $k\geq 2$ and $n$ large we have $\{0,1,\ldots,k\}^n\in\RSU \setminus \bC$.

Now let $P_G$ be the lattice of subgroups of a finite Abelian group $G$. It was shown in~\cite{butler} that $P_G$ is rank-unimodal. The following general question is likely to be difficult to solve in full generality but any progress could be interesting.
\begin{ques}\label{ques1}
For what Abelian groups $G$ is it true that $P_G\in\bC$?
\end{ques}
\noindent
Observe that most results of this paper are special cases of Question~\ref{ques1}:
\begin{itemize}[itemsep=2pt,parsep=2pt,topsep=2pt,partopsep=2pt]
\item if $G=C_{p_1}\times C_{p_2}\times \ldots \times C_{p_n}$ for distinct primes $p_1,p_2,\ldots,p_n$ then $P_G$ is (isomorphic to) the Boolean lattice and hence $P_G\in\bC$,
\item if $G=C_{p_1^k}\times C_{p_2^k}\times \ldots \times C_{p_n^k}$ for distinct primes $p_1,p_2,\ldots,p_n$ then $P_G$ is isomorphic to the lattice $\{0,1,\ldots,k\}^n$ under inclusion and hence if $n\geq n_0(k)$ then $P_G\in\RSU \setminus \bC$.
\item if $G = (C_p)^n $ for $p$ prime then $P_G$ is isomorphic to $\mathcal{V}(p,n)$ and hence $P_G\in\bC$.
\end{itemize}

Question~\ref{ques1} can be asked for other members of $\RSU$, see e.g.~\cite{stanley}.
A natural generalization of the centeredness property is as follows. For an integer $r\geq 2$ say that a poset $P\in\RSU$ has the $r$-\emph{centeredness property} if for all $M$ with $0\leq M \leq |P|$, among all families $\F\subset P$ of size $M$, the number of $r$-chains contained in $P$ is minimized by a centered family. Denote the collection of posets with the $r$-centeredness property by $\bC_r$ and note that $\bC=\bC_2$. A long-standing conjecture in this area due to Kleitman~\cite{kleitman} is that $\{0,1\}^n\in\bC_r$ for all $n,r$. For recent progress on this conjecture we refer the reader to~\cite{baloghwagner, dasgansudakov, griggs}. Asking for a characterisation of $\bC_r$ is currently out of reach, but finding interesting necessary and/or sufficient conditions for a poset $P\in\RSU$ to be in $\bC_r$ could be a fine result.

In a different direction one could improve Theorem~\ref{mainresult} and investigate further for which $M$ Conjecture~\ref{stupidconj} holds.

\begin{ques}\label{ques2}
For which $k$ and $M$ does there exist an $M$-optimal centered family in $\{0,1,\ldots,k\}^n$?
\end{ques}
The same question can be asked for `centered' replaced by `canonical centered' (i.e.~centered families with at most one partially filled layer). We expect that for $k=2$ the answer to Question~\ref{ques2} contains the interval $[0,\Sigma_5(n,2)]$. It seems plausible that for $M\leq \Sigma_{\log_2 n}(n,k)$ the centered families are not too far from being best possible, but for much larger $M$ we do not even have a guess what the best families could be. The following question is open whenever $\sqrt{n}$ is replaced by any value between $\log_2 n$ and $n$.
\begin{ques}
Let $M=\Sigma_{\sqrt{n}}(n,2)$. What do the $M$-optimal families in $\{0,1,2\}^n$ look like?
\end{ques}

\textbf{Note added during the refereeing process:} Very recently Samotij~\cite{samotij} proved some breakthrough result related to the topic of this paper. In particular he proved that the number of $k$-chains in $\{0,1\}^n$ is minimized by centered families.

\textbf{Acknowledgement:} We are very grateful to the referee for spotting an error in the original manuscript and for the many specific and general suggestions they made.

\end{document}